\documentclass[12pt,reqno]{amsart}
\usepackage{amsmath, amsfonts, amssymb, amsthm, hyperref,cite}
\usepackage{bm}
\usepackage{mathrsfs}
\allowdisplaybreaks[4]
\def\l{\left}
\def\r{\right}
\def\bg{\bigg}
\def\({\bg(}
\def\){\bg)}
\def\t{\text}
\def\f{\frac}

\def\eq{\equiv}

\def\Z{\mathbb Z}

\def\N{\mathbb N}

\def\<{\langle}
\def\>{\rangle}
\def\1{{\bf 1}}

\def\qbinom #1#2#3{{\genfrac{[}{]}{0pt}{}{#1}{#2}}_{#3}}
\theoremstyle{plain}
\newtheorem{theorem}{Theorem}

\newtheorem{lemma}{Lemma}
\newtheorem{corollary}{Corollary}
\theoremstyle{definition}

\newtheorem*{Acks}{Acknowledgments}
\theoremstyle{remark}

\numberwithin{equation}{section}

\begin{document}
\title[A refined $q$-analogue of some congruences of Van Hamme]{A refined $q$-analogue of some congruences of Van Hamme}

\author[Chen Wang]{Chen Wang}
\address{Department of Applied Mathematics, Nanjing Forestry University, Nanjing 210037, People's Republic of China}
\email{cwang@smail.nju.edu.cn}

\author[Yu-Chan Tian]{Yu-Chan Tian}
\address{Department of Applied Mathematics, Nanjing Forestry University, Nanjing 210037, People's Republic of China}
\email{tyc@njfu.edu.cn}

\author[Kai Huang]{Kai Huang}
\address{Department of Applied Mathematics, Nanjing Forestry University, Nanjing 210037, People's Republic of China}
\email{kaih@njfu.edu.cn}

\begin{abstract}
In 1997, Van Hamme proposed 13 supercongruences corresponding to $1/\pi$ series of the  Ramanujan-type. Inspired by the recent work of V.J.W. Guo, we establish a unified $q$-analogue of Van Hamme's (B.2), (E.2) and (F.2) supercongruences, which is also a refinement of some known results obtained by different authors. As a consequence, we give a parametric supercongruence related to the Euler polynomials.
\end{abstract}


\maketitle

\section{Introduction}

In 1997, Van Hamme \cite{VH} proposed 13 supercongruences corresponding $1/\pi$ series of the Ramanujan-type, such as
\begin{align}
\sum^{(p-1)/2}_{k=0}(-1)^k(4k+1)\frac{(\frac{1}{2})_k^3}{k!^3} &\equiv p(-1)^{(p-1)/2} \pmod{p^3},\label{b2} \\
\sum_{k=0}^{(p-1)/3} (-1)^k(6k+1)\frac{(\frac{1}{3})_k^3}{k!^3} &\equiv p\pmod{p^3}\quad\text{for}\ p\equiv 1\pmod{3},  \label{e2}\\[5pt]
\sum_{k=0}^{(p-1)/4}(-1)^k(8k+1)\frac{(\frac{1}{4})_k^3}{k!^3} &\equiv p(-1)^{(p-1)/4}\pmod{p^3} \quad \text{for}\ p\equiv 1\pmod{4}, \label{f2}
\end{align}
(tagged (B.2), (E.2), and (F.2), respectively), where $p$ is an odd prime and $(x)_k=x(x+1)\cdots(x+k-1)$ denotes the Pochhammer symbol. All of the 13 supercongruences have been confirmed via using different techniques. In particular, in 2008, the first proof of \eqref{b2} was provided by Mortenson \cite{Mortenson} with the help of Whipple's ${}_7F_6$ transformation and reduction of Gamma functions.  in 2009, \eqref{b2} was reproved by Zudilin \cite{Zudilin} using the WZ (Wilf-Zeilberger) method (cf. \cite{PWZ}). In 2011, \eqref{b2} was reproved again by Long \cite{Long} using a ${}_4F_3$ identity due to Whipple and the $p$-adic Gamma function. Utilizing the same method as the one of Long, Swisher \cite{Swisher} proved \eqref{e2} and \eqref{f2} later.

For any odd prime $p$, it is easy to see that $(1/2)_k\eq0\pmod{p}$ for $k$ in the range $(p+1)/2\leq k\leq p-1$. Therefore, it trivially holds that
$$
\sum^{p-1}_{k=0}(-1)^k(4k+1)\frac{(\frac{1}{2})_k^3}{k!^3}\eq \sum^{(p-1)/2}_{k=0}(-1)^k(4k+1)\frac{(\frac{1}{2})_k^3}{k!^3}\pmod{p^3}.
$$
Surprisingly, Sun \cite{Sun} showed that the above congruence holds modulo $p^4$ for $p>3$. More precisely, employing the WZ method, Sun proved that for any prime $p>3$,
\begin{equation}\label{Suncon}
\sum_{k=0}^{M}(-1)^k(4k+1)\frac{(\frac{1}{2})_k^3}{k!^3}\eq p(-1)^{(p-1)/2}+p^3 E_{p-3}\pmod{p^4},
\end{equation}
where $M=p-1$ or $(p-1)/2$, and $E_{n}$ is the $n$-th Euler number. Clearly, \eqref{Suncon} is a refinement of \eqref{b2}.

Recently, with the help of the WZ method, Guo and Wang \cite{GW2025} established the following parametric supercongruence that gives a unified extension of \eqref{b2}--\eqref{Suncon}: for any prime $p>3$ and $p$-adic integer $\alpha$,
\begin{equation}\label{guowangeq}
\sum_{k=0}^{M}(-1)^k(2k+\alpha)\f{(\alpha)_k^3}{(1)_k^3}\eq (-1)^{\langle -\alpha\rangle_p}(\alpha+\langle -\alpha\rangle_p)+(\alpha+\langle -\alpha\rangle_p)^3E_{p-3}(\alpha)\pmod{p^4},
\end{equation}
where $M=p-1$ or $\<-\alpha\>_p$ (the least nonnegative residue of $-\alpha$ modulo $p$), and $E_n(x)$ is the $n$-th Euler polynomial.

Two continue the story, we need to be familiar with the $q$-notation. The $q$-integer is defined by $[n]=[n]_q=(1-q^n)/(1-q)$, and the $q$-Pochhammer symbol ($q$-shifted factorial) is defined as $(a;q)_n=(1-a)(1-aq)\cdots(1-aq^{n-1})$ for $n=1$ and $(a;q)_0=1$. For the sake of brevity, we adopt the following simplified notation:
$$
(a_1,a_2,\ldots,a_s;q)_n=(a_1;q)_n(a_2;q)_n\cdots(a_s;q)_n.
$$
The $n$-th cyclotomic polynomial is defined as
$$
\Phi_n(q)=\prod_{\substack{1\leq k\leq n\\ \gcd(k,n)=1}}(q-\zeta^k),
$$
where $\zeta=e^{2\pi i/n}$ is an $n$-th primitive root of unity. It is known that $\Phi_n(q)$ is an irreducible polynomial in the ring of polynomials with integer coefficients, and
$$
q^n-1=\prod_{d\mid n}\Phi_d(q).
$$

During the past few years, $q$-analogues of Van Hamme's supercongruences or their generalizations have been widely studied (see, e.g., \cite{G2018,G2019,G2022,GS2020,GuoZu,GuoZu2021,LiuWang,NiWang,Wei2021,Wei2023}). For instance, Guo \cite{G2018} obtained the following $q$-analogue of \eqref{b2}: for any odd integer $n>1$,
\begin{equation}\label{Guob2eq}
\sum_{k=0}^{(n-1)/2}(-1)^kq^{k^2}[4k+1]\f{(q;q^2)_k^3}{(q^2;q^2)_k^3}\eq (-1)^{(n-1)/2}q^{(n-1)^2/4}[n]\pmod{[n]\Phi_n(q)^2}.
\end{equation}
In 2019, Guo \cite{G2019} considered parametric extensions of \eqref{Guob2eq} and gave a unified $q$-analogue of \eqref{b2}--\eqref{f2}. For example, he proved that
$$
\sum_{k=0}^{(n-r)/d}(-1)^kq^{d\binom{k+1}{2}-kr}[2dk+r]\f{(q^r;q^d)_k^3}{(q^d;q^d)_k^3}\eq (-1)^{(n-r)/d}q^{(n-r)(n-d+r)/2d}[n]\pmod{[n]\Phi_n(q)^2},
$$
where $n,d\in\Z^+$, $r\eq n\pmod{d}$ is an integer such that $n+d-nd\leq r\leq n$ and $\gcd(r,d)=1$ when $r\leq 0$. In 2022, Guo \cite{G2022} further studied the $q$-analogue of \eqref{Suncon} and obtained the following $q$-supercongruences: for any positive odd integer $n$,
\begin{align}\label{G2022eq}
\sum_{k=0}^{M}(-1)^k q^{k^2}[4k+1]\f{(q;q^2)_k^3}{(q^2;q^2)_k^3}&\eq (-1)^{(n-1)/2}q^{(1-n^2)/4}\l(q^{\binom{n}{2}}[n]+\f{(n^2-1)(1-q)^2}{24}[n]^3\r)\notag\\
&\quad+[n]^3\sum_{k=1}^{(n-1)/2}\f{q^k(q^2;q^2)_k}{[2k][2k-1](q;q^2)_k}\pmod{[n]\Phi_n(q)^3},
\end{align}
where $M=(n-1)/2$ or $n-1$. Letting $q\to 1$ and $n=p>3$, and using the fact
$$
\sum_{k=1}^{(p-1)/2}\f{4^k}{k(2k-1)\binom{2k}{k}}\eq 2E_{p-3}\pmod{p}
$$
due to Sun \cite{Sun2011}, one immediately obtains \eqref{Suncon}.

Inspired by the above work, it is natural to ask whether \eqref{guowangeq} has a $q$-analogue? The main purpose of this paper is to establish such a $q$-analogue.

\begin{theorem}\label{mainth}
Let $n$ and $d$ be positive integers with $\gcd(n,d)=1$. Let $r$ be an integer with $\gcd(r,d)=1$ and $m$ be the least nonnegative residue of $-r/d$ modulo $n$. Then
\begin{align*}
&\sum_{k=0}^M(-1)^kq^{d\binom{k+1}{2}-kr}[2dk+r]\f{(q^r;q^d)_k^3}{(q^d;q^d)_k^3}\\
&\quad \eq (-1)^mq^{d\binom{m}{2}+mr}\l([md+r]+[md+r]^3\sum_{k=1}^m\f{(-1)^kq^{d\binom{k+1}{2}}(1+q^{dk})}{[kd]^2}\r)\pmod{[n]\Phi_n(q)^3},
\end{align*}
where $M=m$ or $M=n-1$.
\end{theorem}

When $n=1$, Theorem \ref{mainth} holds trivially. Therefore, in the subsequent proofs, we always suppose $n>1$.

Let $p>3$ be a prime and take $n=p$. From \cite{GW2025}, we know
$$
\sum_{k=1}^m\f{(-1)^k}{k^2}\eq\f{(-1)^m}{2}E_{p-3}\l(\f{r}{d}\r)\pmod{p}.
$$
Therefore, letting $q\to 1$ in Theorem \ref{mainth}, we have
$$
\sum_{k=0}^{M}(-1)^k(2dk+r)\f{(\f{r}{d})_k^3}{k!^3}\eq (-1)^m(md+r)+\f{(md+r)^3}{d^2}E_{p-3}\l(\f{r}{d}\r)\pmod{p^4},
$$
where $M=p-1$ or $\<-r/d\>_p$. This means Theorem \ref{mainth} is indeed a $q$-analogue of \eqref{guowangeq}.

When $n$ is a prime power, we can deduce the following result from Theorem \ref{mainth}.

\begin{corollary}
Let $p>3$ be a prime and let $s,\ d$ a positive integer with $p\nmid d$. Let $r$ be an integer with $\gcd(r,d)=1$ and $m$ be the least nonnegative residue of $-r/d$ modulo $p^s$. Then
\begin{equation}\label{Cor1eq1}
\sum_{k=0}^{M}(-1)^k(2dk+r)\f{(\f{r}{d})_k^3}{k!^3}\eq (-1)^m(md+r)+\f{2(-1)^m(md+r)^3}{d^2}\sum_{k=1}^m\f{(-1)^k}{k^2}\pmod{p^{s+3}},
\end{equation}
where $M=p^s-1$ or $m$.
\end{corollary}

In fact, via a similar argument as in the proof of \cite[Lemma 2.1]{WangHu}, we can further reduce the second term on the right-hand side of \eqref{Cor1eq1} to a term involving the Euler polynomials. We leave this work to the interested reader.

We shall prove Theorem \ref{mainth} in next section by using the $q$-WZ method (cf. \cite{PWZ}), together with a $q$-identity proved by the symbolic summation package \verb"Sigma" developed by Schneider \cite{S} in Mathematica.

\section{Proof of Theorem \ref{mainth}}

\begin{lemma}\label{symmetry-3}
Under the conditions of Theorem \ref{mainth}, we have
\begin{align*}
&3d\binom{n-m-1}{2}+d\binom{n-k}{2}+dk-2d-2kr+2r\\
&\quad\eq d\binom{m+k+2}{2}-(m+1+k)r\pmod{n}.
\end{align*}
\end{lemma}

\begin{proof}
Noting that $md\eq -r\pmod{n}$, we obtain
\begin{align*}
&3d\binom{n-m-1}{2}+d\binom{n-k}{2}+dk-2d-2kr+2r-d\binom{m+k+2}{2}+(m+1+k)r\\
&\quad=-2 kr + 3 d m - kd m + d m^2 - 5 d n - kd n - 3 d m n + 2 d n^2 +
 3 r + k r + m r\\
&\quad\eq-2kr-3r+kr-mr+3r+kr+mr\\
&\quad\eq0\pmod{n}.
\end{align*}
This concludes the proof.
\end{proof}

\begin{lemma}\label{modPhi^2}
Under the conditions of Theorem \ref{mainth}, we have
\begin{equation*}
\sum_{k=m+1}^{n-1}(-1)^kq^{d\binom{k+1}{2}-kr}[2dk+r]\f{(aq^r,q^r/a,q^r;q^d)_k}{(aq^d,q^d/a,q^d;q^d)_k}\eq0\pmod{\Phi_n(q)^2},
\end{equation*}
where $a$ is an indeterminate.
\end{lemma}

\begin{proof}
Clearly,
\begin{align}\label{symmetrykey}
&\sum_{k=m+1}^{n-1}(-1)^kq^{d\binom{k+1}{2}-kr}[2dk+r]\f{(aq^r,q^r/a,q^r;q^d)_k}{(aq^d,q^d/a,q^d;q^d)_k}\notag\\
&\quad=\sum_{k=0}^{n-m-2}(-1)^{m+1+k}q^{d\binom{m+k+2}{2}-(m+1+k)r}[2d(m+1+k)+r]\f{(aq^r,q^r/a,q^r;q^d)_{m+1+k}}{(aq^d,q^d/a,q^d;q^d)_{m+1+k}}\notag\\
&\quad=(-1)^{m+1}\f{(aq^r,q^r/a,q^r;q^d)_{m+1}}{(aq^d,q^d/a,q^d;q^d)_{m+1}}\notag\\
&\qquad\cdot\sum_{k=0}^{n-m-2}(-1)^kq^{d\binom{m+k+2}{2}-(m+1+k)r}[2d(m+1+k)+r]\f{(aq^{r+(m+1)d},q^{r+(m+1)d}/a,q^{r+(m+1)d};q^d)_{k}}{(aq^{(m+2)d},q^{(m+2)d}/a,q^{(m+2)d};q^d)_{k}}.
\end{align}
Noting that $(q^r;q^d)_{m+1}$ has the factor $1-q^{r+md}$ and $r+md\eq0\pmod{n}$, we have $(q^r;q^d)_{m+1}\eq0\pmod{\Phi_n(q)}$ and then
\begin{equation}\label{symmetryprod}
(-1)^{m+1}\f{(aq^r,q^r/a,q^r;q^d)_{m+1}}{(aq^d,q^d/a,q^d;q^d)_{m+1}}\eq0\pmod{\Phi_n(q)}.
\end{equation}

Besides, modulo $\Phi_n(q)$ we have
\begin{align}\label{symmetry}
&\sum_{k=0}^{n-m-2}(-1)^kq^{d\binom{m+k+2}{2}-(m+1+k)r}[2d(m+1+k)+r]\f{(aq^{r+(m+1)d},q^{r+(m+1)d}/a,q^{r+(m+1)d};q^d)_{k}}{(aq^{(m+2)d},q^{(m+2)d}/a,q^{(m+2)d};q^d)_{k}}\notag\\
&\eq \sum_{k=0}^{n-m-2}(-1)^kq^{d\binom{m+k+2}{2}-(m+1+k)r}[2d(k+1)-r]\f{(aq^{d},q^{d}/a,q^{d};q^d)_{k}}{(aq^{2d-r},q^{2d-r}/a,q^{2d-r};q^d)_{k}}\notag\\
&\eq(-1)^{n-m-1}\sum_{k=0}^{n-m-2}(-1)^kq^{d\binom{n-k}{2}+(k+1)r+r-2d(k+1)}[2d(k+1)-r]\f{(aq^{d},q^{d}/a,q^{d};q^d)_{n-m-2-k}}{(aq^{2d-r},q^{2d-r}/a,q^{2d-r};q^d)_{n-m-2-k}}.
\end{align}
For $k$ among $0,1,2,\ldots,n-m-2$,
\begin{align}\label{symmetry-1}
&\f{(aq^{d},q^{d}/a,q^{d};q^d)_{n-m-2-k}}{(aq^{2d-r},q^{2d-r}/a,q^{2d-r};q^d)_{n-m-2-k}}\notag\\
&\quad=\f{(aq^{d},q^{d}/a,q^{d};q^d)_{n-m-2}}{(aq^{2d-r},q^{2d-r}/a,q^{2d-r};q^d)_{n-m-2}}\notag\\
&\qquad\cdot\f{q^{3d(kn-km-k(k+1)/2)-3kr}(q^{r-d(n-m-1)}/a,aq^{r-d(n-m-1)},q^{r-d(n-m-1)};q^d)_k}{q^{3d(kn-km-k(k+3)/2)}(q^{(m+2-n)d}/a,aq^{(m+2-n)d},q^{(m+2-n)d};q^d)_k}\notag\\
&\quad\eq \f{(aq^{d},q^{d}/a,q^{d};q^d)_{n-m-2}}{(aq^{2d-r},q^{2d-r}/a,q^{2d-r};q^d)_{n-m-2}}q^{3dk-3kr}\f{(q^d/a,aq^d,q^d;q^d)_k}{(q^{2d-r}/a,aq^{2d-r},q^{2d-r};q^d)_k}\pmod{\Phi_n(q)}.
\end{align}
Meanwhile,
\begin{align}\label{symmetry-2}
&\f{(aq^{d},q^{d}/a,q^{d};q^d)_{n-m-2}}{(aq^{2d-r},q^{2d-r}/a,q^{2d-r};q^d)_{n-m-2}}\notag\\
&\quad\eq\f{(aq^{d},q^{d}/a,q^{d};q^d)_{n-m-2}}{(aq^{(m+2)d},q^{(m+2)d}/a,q^{(m+2)d};q^d)_{n-m-2}}\notag\\
&\quad=\f{\prod_{j=1}^{n-m-2}(1-aq^{dj})(1-q^{dj}/a)(1-q^{dj})}{\prod_{j=1}^{n-m-2}(1-q^{(m+j+1)d}/a)(1-aq^{(m+j+1)d})(1-q^{(m+j+1)d})}\notag\\
&\quad=\f{\prod_{j=1}^{n-m-2}(1-aq^{dj})(1-q^{dj}/a)(1-q^{dj})}{\prod_{j=1}^{n-m-2}(1-q^{(n-m-1-j+m+1)d}/a)(1-aq^{(n-m-1-j+m+1)d})(1-q^{(n-m-1-j+m+1)d})}\notag\\
&\quad\eq\f{\prod_{j=1}^{n-m-2}(1-aq^{dj})(1-q^{dj}/a)(1-q^{dj})}{\prod_{j=1}^{n-m-2}(1-q^{-jd}/a)(1-aq^{-jd})(1-q^{-jd})}\notag\\
&\quad=(-1)^{n-m}q^{3d\binom{n-m-1}{2}}\pmod{\Phi_n(q)}.
\end{align}
Combining \eqref{symmetry}--\eqref{symmetry-2}, we conclude that
\begin{align*}
&\sum_{k=0}^{n-m-2}(-1)^kq^{d\binom{m+k+2}{2}-(m+1+k)r}[2d(m+1+k)+r]\f{(aq^{r+(m+1)d},q^{r+(m+1)d}/a,q^{r+(m+1)d};q^d)_{k}}{(aq^{(m+2)d},q^{(m+2)d}/a,q^{(m+2)d};q^d)_{k}}\\
&\quad\eq0\pmod{\Phi_n(q)}.
\end{align*}
This, together with \eqref{symmetrykey}, \eqref{symmetryprod} and Lemma \ref{symmetry-3}, gives Lemma \ref{modPhi^2}.
\end{proof}

\begin{lemma}\label{liangduan-a}
Under the conditions of Theorem \ref{mainth}, modulo $[n]\Phi_n(q)(a-q^{md+r})(1-aq^{md+r})$, we have
\begin{align}\label{liangduan-aeq}
&\sum_{k=0}^{m}(-1)^kq^{d\binom{k+1}{2}-kr}[2dk+r]\f{(aq^r,q^r/a,q^r;q^d)_k}{(aq^d,q^d/a,q^d;q^d)_k}\notag\\
&\quad\eq \sum_{k=0}^{n-1}(-1)^kq^{d\binom{k+1}{2}-kr}[2dk+r]\f{(aq^r,q^r/a,q^r;q^d)_k}{(aq^d,q^d/a,q^d;q^d)_k},
\end{align}
where $a$ is an indeterminate.
\end{lemma}

\begin{proof}
By \cite[Lemma 2.2]{GS2020}, under the conditions of Theorem \ref{mainth}, we have
\begin{align*}
&\sum_{k=0}^{m}[2dk+r]\f{(aq^r,q^r/a,q^r/b,q^r;q^d)_k}{(aq^d,q^d/a,bq^d,q^d;q^d)_k}b^kq^{(d-2r)k}\\
&\quad\eq \sum_{k=0}^{n-1}[2dk+r]\f{(aq^r,q^r/a,q^r/b,q^r;q^d)_k}{(aq^d,q^d/a,bq^d,q^d;q^d)_k}b^kq^{(d-2r)k}\pmod{[n]},
\end{align*}
where $a$ and $b$ are indeterminates. Letting $b\to 0$ in the above congruence, we obtain
\begin{align*}
&\sum_{k=0}^{m}(-1)^kq^{d\binom{k+1}{2}-kr}[2dk+r]\f{(aq^r,q^r/a,q^r;q^d)_k}{(aq^d,q^d/a,q^d;q^d)_k}\notag\\
&\quad\eq \sum_{k=0}^{n-1}(-1)^kq^{d\binom{k+1}{2}-kr}[2dk+r]\f{(aq^r,q^r/a,q^r;q^d)_k}{(aq^d,q^d/a,q^d;q^d)_k}\notag\\
&\quad\eq 0\pmod{[n]}.
\end{align*}

Recall that Jackson's $_6\phi_5$ summation formula (cf. \cite[Appendix (II.21)]{GR}) reads
\begin{equation}\label{Jackson}
\sum_{k=0}^N\f{(1-aq^{2k})(a,b,c,q^{-N};q)_k}{(q,aq/b,aq/c,aq^{N+1};q)_k}\l(\f{aq^{N+1}}{bc}\r)^k=\f{(1-a)(aq,aq/bc;q)_N}{(aq/b,aq/c;q)_N}.
\end{equation}
Letting $N\to\infty,\ q\to q^d,\ a=q^r,\ b=aq^r,\ c=q^r/a$ in \eqref{Jackson}, we have
\begin{equation}\label{Jacksonsubs}
\sum_{k=0}^{\infty}(-1)^kq^{d\binom{k+1}{2}-kr}[2dk+r]\f{(aq^r,q^r/a,q^r;q^d)_k}{(aq^d,q^d/a,q^d;q^d)_k}=\f{[r](q^{d-r},q^{d+r};q^d)_{\infty}}{(aq^{d},q^d/a;q^d)_{\infty}}.
\end{equation}
Taking $a=q^{md+r}$ or $a=q^{-md-r}$ in \eqref{Jacksonsubs}, we arrive at
\begin{equation}\label{Jacksonsubstrunc}
\sum_{k=0}^{\infty}(-1)^kq^{d\binom{k+1}{2}-kr}[2dk+r]\f{(q^{-md},q^{md+2r},q^r;q^d)_k}{(q^{md+d+r},q^{d-md-r},q^d;q^d)_k}=\f{[r](q^{d-r},q^{d+r};q^d)_{\infty}}{(q^{d+md+r},q^{d-md-r};q^d)_{\infty}}.
\end{equation}
Note that $(q^{-md};q^d)_k=0$ for $k>m$, and $(q^{md+d+r};q^d)_k\neq0,\ (q^{d-md-r};q^d)_k\neq0$ for any $k\in\N$. Moreover, it is easy to see that
$$
\f{(q^{d-r},q^{d+r};q^d)_{\infty}}{(q^{d+md+r},q^{d-md-r};q^d)_{\infty}}=\f{(q^{d+r};q^d)_m}{(q^{d-md-r;q^d})_m}.
$$
Therefore, \eqref{Jacksonsubstrunc} actually reduces to
$$
\sum_{k=0}^{M}(-1)^kq^{d\binom{k+1}{2}-kr}[2dk+r]\f{(q^{-md},q^{md+2r},q^r;q^d)_k}{(q^{md+d+r},q^{d-md-r},q^d;q^d)_k}=\f{[r](q^{d+r};q^d)_m}{(q^{d-md-r;q^d})_m},
$$
that is, both sides of \eqref{liangduan-aeq} are equal when $a=q^{md+r}$ or $a=q^{-md-r}$. Thus, the $q$-congruence is true modulo $(a-q^{md+r})(1-aq^{md+r})$.

In view of the above and Lemma \ref{modPhi^2}, we conclude the proof.
\end{proof}

By Lemma \ref{liangduan-a} with $a\to 1$ and noting that $1-q^{md+r}$ contains the factor $\Phi_n(q)$, we immediately obtain the following result.
\begin{lemma}\label{truncon}
Under the conditions of Theorem \ref{mainth}, we have
\begin{align*}
&\sum_{k=0}^m(-1)^kq^{d\binom{k+1}{2}-kr}[2dk+r]\f{(q^r;q^d)_k^3}{(q^d;q^d)_k^3}\\
&\quad\eq\sum_{k=0}^{n-1}(-1)^kq^{d\binom{k+1}{2}-kr}[2dk+r]\f{(q^r;q^d)_k^3}{(q^d;q^d)_k^3}\pmod{[n]\Phi_n(q)^3}.
\end{align*}
\end{lemma}

\begin{lemma}\label{A(q)/B(q)}
Let $r,d,n,m$ be defined as in Theorem \ref{mainth}. Write
$$
\f{(q^r;q^d)_m(q^{r+(m+1)d};q^d)_m}{(q^d;q^d)_m^2}=\f{A(q)}{B(q)},
$$
where $A(q)$ and $B(q)$ are relatively prime polynomials in $q$. Then $B(q)$ is relatively prime to $1-q^n$.
\end{lemma}

\begin{proof}
Clearly,
$$
q^k-1=\begin{cases}\prod_{t\mid k}\Phi_t(q),\quad&\t{if}\ k>0,\\ -q^k\prod_{t\mid k}\Phi_t(q),\quad&\t{if}\ k<0.\end{cases}
$$
Hence, we can write
\begin{align*}
(q^r;q^d)_m&=\pm q^u\prod_{t=1}^{\infty}\Phi_t(q)^{f_t},\\
(q^{r+(m+1)d};q^d)_m&=\pm q^v\prod_{t=1}^{\infty}\Phi_t(q)^{g_t},\\
(q^d;q^d)_m^2&=\prod_{t=1}^{\infty}\Phi_t(q)^{2h_t},
\end{align*}
where $u,v$ are integers, and $f_t,\ g_t$ and $h_t$ denote the numbers of multiples of $t$ in the sets $\{r,r+d,\ldots,r+(m-1)d\},\ \{r+(m+1)d,r+(m+2)d,\ldots,r+2md\}$ and $\{d,2d,\ldots,md\}$, respectively.

Assume $t\mid n$, then $t\mid md+r$. Since $\gcd(d,n)=1$, we have $\gcd(d,t)=1$. Therefore, for each $k\in\N$,
$$
S_k=\{r+ktd,r+(kt+1)d,\ldots,r+(kt+t-1)d\}
$$
is a complete set of residues modulo $t$. This means that for each $k\in\N$, there is exactly one multiple of $t$ in $S_k$. Below we consider two cases.

{\it Case 1}. $t\mid r$.

In this case, $f_t=\lfloor m/t\rfloor$, where $\lfloor x\rfloor$ stands for the greatest integer not exceeding $x$. Similarly, we also have $g_t=h_t=\lfloor m/t\rfloor$. Therefore, $f_t+g_t\geq 2h_t$.

{\it Case 2}. $t\nmid r$.

Now, $f_t=\lfloor (m-1)/t\rfloor$ or $f_t=\lfloor (m-1)/t\rfloor+1$. Since $t\nmid r$ and $t\mid md+r$, we have $t\nmid m$ and $g_t=h_t=\lfloor m/t\rfloor=\lfloor (m-1)/t\rfloor$. Thus, $f_t+g_t\geq 2h_t$.

In view of the above, we are done.
\end{proof}

\begin{lemma}\label{Akq/Bkq}
Let $r,d,n,m$ be defined as in Theorem \ref{mainth}. For $k\in\{1,2,\ldots,m\}$, write
$$
\f{1-q^{md+r}}{1-q^{dk}}=\f{A_k(q)}{B_k(q)},
$$
where $A_k(q)$ and $B_k(q)$ are relatively prime polynomials in $q$. Then $B_k(q)$ is relatively prime to $1-q^n$.
\end{lemma}

\begin{proof}
Since $n\mid md+r$, $1-q^{md+r}$ contains all irreducible factors of $1-q^n$. Then we arrive at the desired result by noting that $q^{dk}-1$ can be written as a product of different cyclotomic polynomials.
\end{proof}

\begin{lemma}\label{Ckq/Dkq}
Let $r,d,n,m$ be defined as in Theorem \ref{mainth}. For $k=1,2,\ldots,m$, write
$$
\f{[md+r](q^d;q^d)_{k-1}}{(q^{-md};q^d)_k}=\f{C_k(q)}{D_k(q)},
$$
where $C_k(q)$ and $D_k(q)$ are relatively prime polynomials in $q$. Then $D_k(q)$ is relatively prime to $1-q^n$.
\end{lemma}

\begin{proof}
Write
\begin{align*}
(q^d;q^d)_{k-1}&=\pm\prod_{t=1}^{\infty}\Phi_t(q)^{\alpha_t},\\
(q^{-md};q^d)_k&=\pm q^w\prod_{t=1}^{\infty}\Phi_t(q)^{\beta_t},
\end{align*}
where $w$ is an integer, and $\alpha_t,\ \beta_t$ are the numbers of multiples of $t$ in the sets $\{d,2d,\ldots,(k-1)d\},\ \{-md,-md+d,\ldots,-md+(k-1)d\}$, respectively.

Assume $t\mid n$. Similarly as before, we have $\alpha_t=\lfloor(k-1)/t\rfloor$ and $\beta_t\leq \lfloor(k-1)/t\rfloor+1$, which implies $\alpha_t-\beta_t\geq -1$. Then the desired follows from the facts $n\mid md+r$ and $1-q^{md+r}=\pm q^{x}\prod_{t\mid md+r}\Phi_t(q)$, where $x$ is an integer.
\end{proof}

For given numbers $d$ and $r$, as in \cite{G2019}, define the rational functions:
\begin{align*}
F(k,l)&=(-1)^{k+l}q^{d\binom{k-l+1}{2}-rk+rl}[2dk+r]\f{(q^r;q^d)_k^2(q^r;q^d)_{k+l}}{(q^d;q^d)_k^2(q^d;q^d)_{k-l}(q^r;q^d)_l^2},\\
G(k,l)&=(-1)^{k+l}q^{d\binom{k-l+1}{2}-rk+rl}\f{(q^r;q^d)_k^2(q^r;q^d)_{k+l-1}}{(1-q)(q^d;q^d)_{k-1}^2(q^d;q^d)_{k-l}(q^r;q^d)_l^2},
\end{align*}
where $1/(q^d;q^d)_k=0$ for any negative integer $k$. These two functions form a $q$-WZ pair (cf. \cite{PWZ}). That is, for any $k\in\N$ and $l\in\Z^+$,
\begin{equation}\label{WZ}
F(k,l-1)-F(k,l)=G(k+1,l)-G(k,l).
\end{equation}

\begin{lemma}\label{Fmm}
Under the conditions of Theorem \ref{mainth}, we have
\begin{equation}\label{Fmmeq}
F(m,m)\eq (-1)^mq^{d\binom{m}{2}+mr}[md+r]\l(1-[md+r]^2\sum_{k=1}^m\f{q^{dk}}{[dk]^2}\r)\pmod{[n]\Phi_n(q)^3}.
\end{equation}
\end{lemma}

\begin{proof}
Clearly,
$$
F(m,m)=[2md+r]\f{(q^r;q^d)_{2m}}{(q^d;q^d)_m^2}=[md+r]\f{(q^r;q^d)_m(q^{r+(m+1)d};q^d)_m}{(q^d;q^d)_m^2}.
$$
In light of Lemmas \ref{A(q)/B(q)}, \ref{Akq/Bkq} and the fact that $[n]\mid [md+r]$, both sides of \eqref{Fmmeq} are congruent modulo $[n]$. It remains to show \eqref{Fmmeq} holds modulo $\Phi_n(q)^4$.

Since $q^{md+r}\eq1\pmod{\Phi_n(q)}$, we may write
$$
q^{md+r}=1+t(q)\Phi_n(q),
$$
where $t(q)$ is a rational function in $q$ whose denominator is relatively prime to $\Phi_n(q)$. Therefore,
$$
q^{-md-r}=\f{1}{1+t(q)\Phi_n(q)}\eq 1-t(q)\Phi_n(q)+t(q)^2\Phi_n(q)^2\pmod{\Phi_n(q)^3},
$$
and then we have
\begin{align*}
&(q^r;q^d)_m(q^{md+r+d};q^d)_m\\
&\quad=(-1)^mq^{d\binom{m}{2}+mr}\prod_{j=1}^m(1-q^{-md-r+jd})(1-q^{md+r+jd})\\
&\quad\eq (-1)^mq^{d\binom{m}{2}+mr}\prod_{j=1}^m (1-q^{jd}-q^{jd}t(q)\Phi_n(q))(1-q^{jd}+q^{jd}t(q)\Phi_n(q)-q^{jd}t(q)^2\Phi_n(q)^2)\\
&\quad\eq (-1)^mq^{d\binom{m}{2}+mr}\prod_{j=1}^m\l((1-q^{jd})^2-q^{jd}t(q)^2\Phi_n(q)^2\r)\\
&\quad\eq (-1)^mq^{d\binom{m}{2}+mr}\prod_{j=1}^m(1-q^{jd})^2\l(1-t(q)^2\Phi_n(q)^2\sum_{k=1}^m\f{q^{dk}}{(1-q^{dk})^2}\r)\\
&\quad= (-1)^mq^{d\binom{m}{2}+mr}(q^d;q^d)_m^2\l(1-[md+r]^2\sum_{k=1}^m\f{q^{dk}}{[dk]^2}\r)\pmod{\Phi_n(q)^3}.
\end{align*}
Noticing that $[md+r]\eq0\pmod{\Phi_n(q)}$, we obtain
$$
F(m,m)\eq [md+r](-1)^mq^{d\binom{m}{2}+mr}\l(1-[md+r]^2\sum_{k=1}^m\f{q^{dk}}{[dk]^2}\r)\pmod{\Phi_n(q)^4}.
$$
This concludes the proof.
\end{proof}

\begin{lemma}\label{G(m+1,k)}
Under the conditions of Theorem \ref{mainth}, for $k=1,2,\ldots,m$, we have
\begin{equation}\label{G(m+1,k)eq}
G(m+1,k)\eq \f{q^{-mdk}[md+r]^3(q^d;q^d)_k}{[dk][dk-(m+1)d](q^{-md};q^d)_k}\pmod{[n]\Phi_n(q)^3}.
\end{equation}
\end{lemma}

\begin{proof}
By Lemma \ref{Ckq/Dkq} and via a similar argument as the proof of Lemma \ref{Akq/Bkq}, we know that the right-hand side of \eqref{G(m+1,k)eq} is congruent to $0$ modulo $[n]$. On the other hand,
\begin{align}\label{G(m+1,k)reduce}
G(m+1,k)&=(-1)^{m+1+k}q^{d\binom{m+2-k}{2}-r(m+1)+rk}\f{(q^r;q^d)_{m+1}^2(q^r;q^d)_{m+k}}{(1-q)(q^d;q^d)_m^2(q^d;q^d)_{m+1-k}(q^r;q^d)_k^2}\notag\\
&=(-1)^{m+1+k}q^{d\binom{m+2-k}{2}-r(m+1)+rk}\f{(1-q^{md+r})^3(q^r;q^d)_m^3}{(1-q)(q^d;q^d)_m^2}\notag\\
&\quad\times\f{(q^{(m+1)d+r};q^d)_{k-1}}{(q^d;q^d)_{m+1-k}(q^r;q^d)_k^2}\notag\\
&=(-1)^{m+1+k}q^{d\binom{m+2-k}{2}-r(m+1)+rk}\f{(1-q^{md+r})^3(q^r;q^d)_m^3}{(1-q)(q^d;q^d)_m^3}\notag\\
&\quad\times\f{(q^{(m+1)d+r};q^d)_{k-1}(q^{(m-k+2)d};q^d)_{k-1}}{(q^r;q^d)_k^2}.
\end{align}
Via similar arguments as the ones in the proofs of Lemmas \ref{A(q)/B(q)} and \ref{Ckq/Dkq}, it is not hard to see $G(m+1,k)\eq0\pmod{[n]}$. Therefore, both sides of \eqref{G(m+1,k)eq} are congruent modulo $[n]$.

It suffices to show \eqref{G(m+1,k)eq} holds modulo $\Phi_n(q)^4$. Note that
$$
\f{(q^r;q^d)_m^3}{(q^d;q^d)_m^3}\eq \f{(q^{-md};q^d)_m^3}{(q^d;q^d)_m^3}=\l(\prod_{j=1}^{m}\f{1-q^{-md+(m+1-j-1)d}}{1-q^{jd}}\r)^3=(-1)^mq^{-3d\binom{m+1}{2}}\pmod{\Phi_n(q)}
$$
and
$$
(q^{(m-k+2)d};q^d)_{k-1}=\prod_{j=1}^{k-1}(1-q^{(m-k+j+1)d})=\prod_{j=1}^{k-1}(1-q^{(m-j+1)d})=(-1)^{k-1}q^{md(k-1)-d\binom{k-1}{2}}(q^{-md};q^d)_{k-1}.
$$
Thus, by \eqref{G(m+1,k)reduce}, we have
\begin{align*}
G(m+1,k)&\eq q^{d\binom{m+2-k}{2}-3d\binom{m+1}{2}-d\binom{k-1}{2}-rm}\f{(1-q^{md+r})^3(q^d;q^d)_{k-1}(q^{-md};q^d)_{k-1}}{(1-q)(q^{-md};q^d)_{k}^2}\\
&=q^{d\binom{m+2-k}{2}-3d\binom{m+1}{2}-d\binom{k-1}{2}-rm}\f{[md+r]^3(q^d;q^d)_k}{[dk][dk-(m+1)d](q^{-md};q^d)_k}\pmod{\Phi_n(q)^4}.
\end{align*}
Since
$$
d\binom{m+2-k}{2}-3d\binom{m+1}{2}-d\binom{k}{2}-rm=-mdk -  m^2d - m r\eq-mdk\pmod{n},
$$
\eqref{G(m+1,k)eq} holds modulo $\Phi_n(q)^4$. This proves the desired result.
\end{proof}

\begin{lemma}\label{identity}
For any positive integer $n$, we have
\begin{equation}\label{identityeq}
(-1)^nq^{\binom{n+1}{2}}\sum_{k=1}^n\f{q^{-nk}(q;q)_k}{[k][k-n-1](q^{-n};q)_k}=\sum_{k=1}^n\f{(-1)^kq^{\binom{k+1}{2}}(1+q^k)}{[k]^2}+\sum_{k=1}^n\f{q^k}{[k]^2}.
\end{equation}
\end{lemma}

\begin{proof}
Recall that the $q$-binomial coefficient
$$\qbinom{n}{k}{q}=\f{(q;q)_n}{(q;q)_k(q;q)_{n-k}}.$$
It is easy to see that
\begin{align*}
\f{(q;q)_k}{(q^{-n};q)_k}=\f{(q;q)_k(q^{-n+k};q)_{n-k}}{(q^{-n};q)_n}=\f{(-1)^kq^{((2n+1)k-k^2)/2} (q;q)_k(q;q)_{n-k}}{(q;q)_n}=\f{(-1)^kq^{((2n+1)k-k^2)/2}}{\qbinom{n}{k}{q}}.
\end{align*}
Therefore, we may write the left-hand side of \eqref{identityeq} as
$$
S_n:=(-1)^nq^{\binom{n+1}{2}}\sum_{k=1}^{n}\f{(-1)^kq^{(1-k)k/2}}{[k][k-n-1]\qbinom{n}{k}{q}}.
$$
Using the symbolic summation package \verb"Sigma" developed by Schneider \cite{S} in Mathematica, we find $S_n$ satisfies the following recurrence relation:
$$
S_{n+1}-S_n=\f{q^{n+1}}{[n+1]^2}+\f{(-1)^{n+1}q^{\binom{n+2}{2}}(1+q^{n+1})}{[n+1]^2}.
$$
Clearly, the sum on the right-hand side of \eqref{identityeq} satisfies the same recurrence relation. Moreover, it can be directly verified that the both sides of \eqref{identityeq} are equal when $n=1$. This concludes the proof.
\end{proof}

\noindent{\it Proof of Theorem \ref{mainth}}. In view of Lemma \ref{truncon}, we only need to prove Theorem \ref{mainth} for $M=m$. Summing both sides of \eqref{WZ} over $k$ from $0$ to $m$, we get
\begin{equation}\label{WZ1}
\sum_{k=0}^{m}F(k,l-1)-\sum_{k=0}^mF(k,l)=G(m+1,l)-G(0,l)=G(m+1,l).
\end{equation}
Then, summing \eqref{WZ1} over $l$ from $1$ to $m$ and observing that $F(k,m)=0$ for $k<m$, we arrive at
\begin{equation}\label{WZkey}
\sum_{k=0}^mF(k,0)=F(m,m)+\sum_{k=1}^mG(m+1,k).
\end{equation}
By Lemma \ref{G(m+1,k)} and Lemma \ref{identity} with $q\to q^d$, modulo $[n]\Phi_n(q)^3$, we obtain
$$
\sum_{k=1}^mG(m+1,k)\eq (-1)^mq^{-d\binom{m+1}{2}}[md+r]^3\l(\sum_{k=1}^m\f{(-1)^kq^{d\binom{k+1}{2}}(1+q^{dk})}{[dk]^2}+\sum_{k=1}^m\f{q^{dk}}{[dk]^2}\r).
$$
Then, by Lemma \ref{Fmm}, we finally completes the proof.\qed

\begin{Acks}
This work is partially supported by the National Natural Science Foundation of China (grant 12201301) and the Jiangsu Province Student Innovation Training Program (grant 202410298185Y).
\end{Acks}

\end{document}